\documentclass[10pt]{article}
\usepackage{amstext}
\usepackage{amssymb}
\usepackage{amsmath}
\usepackage{amsthm}
\usepackage[dvips]{graphicx}        % макропакет для включения рисунков eps командой \includegraphicx{XXX.eps}
\graphicspath{{./pict/}}            % Путь к папке с рисунками
\usepackage{color}

\usepackage[cp1251]{inputenc}
\usepackage[russian,english]{babel}

\theoremstyle{plain}
\newtheorem{theorem}{Theorem}[section]

\newtheorem{corol}[theorem]{Corollary}
\newtheorem{conjecture}[theorem]{Conjecture}

\theoremstyle{definition}
\newtheorem{definition}[theorem]{Definition}

\newtheorem{remark}[theorem]{Remark}

\DeclareMathOperator{\PR}{P\mathbb R}
\DeclareMathOperator{\Ind}{Ind}

\def\sgn{\mathop{\rm sign}\nolimits}
\newcommand{\IndI}{\Ind_\infty}
\newcommand{\IndC}{\Ind_{-\infty}^{+\infty}}
\newcommand{\IndPR}{\Ind_{\PR}}

\def\sgn{\mathop{\rm sign}\nolimits}
\def\eqbd{\mathop{{:}{=}}}

\begin{document}
\vspace{\baselineskip}

\vspace{\baselineskip} \thispagestyle{empty}

\title{Hawaii conjecture through the lens of Cauchy indices}

\author{
Dmitry Gnatyuk and %\thanks{Research was partially supported by
%National Science Foundation of China, grant no. BC0710062.}
%\\
%\small Department of Science and Mathematics\\
%\small Wheelock College, USA\\
%\small 800 Dong Chuan Road, 200240,\\
%\small Shanghai, P.R. China\\
%\small Email: {\tt olga.m.katkova@gmail.com}
% \and
Mikhail Tyaglov%\thanks{The work of M.\,Tyaglov was supported with the support of the Russian Science Foundation Grant 19-71-30002}
%National Natural Science Foundation
%of China under grant no.~11871336.}
\\
%\small School of Mathematical Sciences and MOE-LSC\\
%\small Shanghai Jiao Tong University\\
%\small 800 Dong Chuan Road, 200240,\\
%\small Shanghai, 200240, P.R. China\\
%\small and \\
\small Department of Mathematics and Computer Sciences\\
\small Saint Petersburg State University, Saint Petersburg, 199178, Russia\\
\small Email: {\tt tyaglov@mail.ru}
%\and
%Anna Vishnyakova
%\\
%\small School of Mathematics and Computer Sciences\\
%\small Kharkov National V.N.Karazin University\\
%\small Kharkov, Ukraine\\
%\small Email: {\tt anna.m.vishnyakova@univer.kharkov.ua}
}

\date{\small \today}

\maketitle

\vspace{5mm}


\vspace{2mm}


\vspace{4mm}

\begin{flushright}
\textit{Dedicated to Alexander Aptekarev on the\\ occasion of his 70$^{th}$ anniversary.}
\end{flushright}

\vspace{2mm}

\begin{abstract}
Given a real polynomial $p$, we study some properties of real critical points of its logarithmic derivative
$Q[p]=(p'/p)'$ using the theory of Cauchy indices. As a by-product we improve the lower bound for the number
these points.
\end{abstract}

%\keywords{Hermite-Biehler theorem, root localization}
%\subjclass[2000]{Primary 26C05, Secondary 30C15}

%%%%%%%%%%%%%%%%%%%%%%%%%%%%%%%%%%%%%%%
\section{Introduction}
%%%%%%%%%%%%%%%%%%%%%%%%%%%%%%%%%%%%%%%

In the remarkable work~\cite{CravenCsordasSmith} T.\,Craven, G.\,Csordas, and W.\,Smith posed a conjecture stated that for a real polynomial $p$ the number of the
real critical point of its logarithmic derivative
\begin{equation}\label{main.function.2}
Q[p](z)\stackrel{def}{=}\dfrac{d}{dz}\left(\dfrac{p'(z)}{p(z)}\right)=
\dfrac{p(z)p''(z)-\left(p'(z)\right)^2}{\left(p(z)\right)^2}.
\end{equation}
does not exceed the number of non-real zeroes of $p$:
\begin{equation}\label{Hawaii.conjecture}
Z_{\mathbb{R}}\left(Q[p]\right)\leqslant2m,
\end{equation}
where $2m=Z_{\mathbb{C}}(p)$, and $Z_{\mathbb{R}}(f)$ and $Z_{\mathbb{C}}(f)$ denote the number of real and non-real zeroes of a function $f$, respectively.

In~\cite{Tyaglov_Hawaii} the Hawaii conjecture was proved  in a rather sophisticated way for polynomials and entire functions of class $\mathcal{L-P}^*$
(see, e.g.,~\cite[Definition~1.4]{Tyaglov_Hawaii} and references there). The method used in~\cite{Tyaglov_Hawaii} seemed not to be simplified for polynomials
only. However, it remains interesting to find a more simple proof based on polynomial properties, since once the Hawaii conjecture is true for polynomials,
it is true for some classes of entire functions as well~\cite{Csordas}.
We remind the reader that the proof of the Hawaii conjecture was based on the following inequalities~\cite[Theorem~3.16]{Tyaglov_Hawaii}
\begin{equation}\label{Number.real.zeroes.Q.ineq.Q1}
2m-2m_{1}\leqslant Z_{\mathbb{R}}\left(Q[p]\right)\leqslant2m-2m_{1}+Z_{\mathbb{R}}\left(Q[p']\right)
\end{equation}
hold for polynomials and entire functions of class $\mathcal{L-P}^*$ possessing the so-called property~A (see~\cite[Definition~1.26]{Tyaglov_Hawaii}).
The property is rather technical, and functions with such a property seem to be not of interest themselves, out of that proof. However, any real
polynomial or entire function~$f$ can be replaced by a new function $g$ such that $Q[f]=Q[g]$ and $g$ possesses property~A. Precisely,
$g(z)=e^{\sigma z}f(z)$, where an appropriate $\sigma\in\mathbb{R}$ always exists,~\cite[Theorem~1.28]{Tyaglov_Hawaii}. Basing on this fact,
one can reformulate~\cite[Theorem~3.16]{Tyaglov_Hawaii} without any additional properties for arbitrary polynomials and entire functions
of class~$\mathcal{L-P}^*$ (see Theorem~\ref{Theorem.Hawaii.base.ineq} below).

In what follows, we restrict ourselves by considering only real polynomial. However, some facts can be extended to the class~$\mathcal{L-P}^*$
of entire functions.

\begin{definition}
Given a polynomial $p$ and a real number $\sigma$, we call the polynomial $p_{\sigma}$ defined as
\begin{equation}\label{p.sigma}
p_{\sigma}(z)\stackrel{def}{=}p'(z)+\sigma p(z),
\end{equation}
the generalised or Laguerre derivative of  $p$. It is clear that $p_{0}$
coincides with the ordinary derivative $p'$.
\end{definition}

For the Laguerre derivative, the following inequality holds (see, e.g., \cite[Problem 192, p.~69]{Polya&Szego}).
\begin{equation}\label{Propos.sigma.p.zeroes}
Z_{\mathbb{C}}(p_{\sigma})\leqslant Z_{\mathbb{C}}(p), \qquad\forall \sigma\in\mathbb{R}
\end{equation}
%\
where $Z_{\mathbb{C}}(f)$ denotes the number of non-real zeroes of the function $f$.

Thus, since $g'(z)=e^{\sigma z}f_{\sigma}(z)$, Theorem~3.16 from~\cite{Tyaglov_Hawaii} can be reformulated for polynomials as follows.

\begin{theorem}\label{Theorem.Hawaii.base.ineq}
Let $p$ be a real polynomial with $Z_{\mathbb{C}}(p)=2m>0$. Then there exists $\sigma\in\mathbb{R}$ such that
\begin{equation*}\label{Number.real.zeroes.Q.ineq.Q_sigma}
2m-2m_{\sigma}\leqslant Z_{\mathbb{R}}\left(Q[p]\right)\leqslant2m-2m_{\sigma}+Z_{\mathbb{R}}\left(Q[p_{\sigma}]\right),
\end{equation*}
where  $Q[p_{\sigma}]=\left({p'_{\sigma}}/{p_{\sigma}}\right)'$, and $2m_{\sigma}=Z_{\mathbb{C}}\left(p_{\sigma}\right)<2m$.
%the polynomial $p_{\sigma}$ is defined in~\eqref{p.sigma}.
\end{theorem}

This theorem, in fact, allows us to apply the method of Rolle's theorem from~\cite{Tyaglov_Hawaii} for polynomials only.
However, this does not simplify the proof much, because the proof from~\cite{Tyaglov_Hawaii} for polynomials is the longest
one. The shortest case is the entire functions without maximal and minimal real zeroes, e.g. sine multiplied by a polynomials.

In this work, we present another approach to study the properties and the number of real zeroes of the function $Q[p]$,
the so-called Cauchy indices introduced by A.\,Cauchy in 1837. In Section~\ref{section:Zeroes.of.Q.and.Q_sigma}, we find the
connection between zeroes $Q[p]$ and $Q[p_{\sigma}]$. Section~\ref{section:Zeroes.of.Q.and.p_sigma} is devoted to
a relation between zeroes of $Q[p]$ and $p_{\sigma}$. Here we obtain that, given a real polynomial $p$, there exists only
finitely many complex numbers which are multiple zeroes of the polynomials of the family $p_{\sigma}$ (Corollary~\ref{Corol.finite.number.multiple roots}).
Finally, in Section~\ref{section:Cauchy.indices} using the Cauchy indices, we establish a generalisation (one direction)
of~\cite[Theorem~1 and Corollary~1]{CravenCsordasSmith}. As a by-product we improve the lower bound for the number
$Z_{\mathbb{R}}\left(Q[p]\right)$, cf.~Theorem~\ref{Theorem.Hawaii.base.ineq}.

\setcounter{equation}{0}

%%%%%%%%%%%%%%%%%%%%%%%%%%%%%%%%%%%%%%%%%%%%%%%%%%%%%%%%%%%%%%%%%%%%%%%%%%%%%%%%%%%%%%%%%%%%%%%%%%%%%%%%%%%%%%%%%%%%
\section{Critical points of logarithmic derivatives of $p$ and $p_{\sigma}$}\label{section:Zeroes.of.Q.and.Q_sigma}
%%%%%%%%%%%%%%%%%%%%%%%%%%%%%%%%%%%%%%%%%%%%%%%%%%%%%%%%%%%%%%%%%%%%%%%%%%%%%%%%%%%%%%%%%%%%%%%%%%%%%%%%%%%%%%%%%%%%

We remind that the critical points of the logarithmic derivatives of a given real polynomial $p$ are the zeroes of
the function
$$
Q[p](z)=\left(\dfrac{p'(z)}{p(z)}\right)'=\dfrac{p(z)p''(z)-[p'(z)]^2}{p^2(z)}.
$$

Together with the function $Q[p]$, we will study the following polynomial
\begin{equation}\label{F.functions}
F[p](z)=p(z)p''(z)-[p'(z)]^2,
\end{equation}
which is the numerator of the rational function $Q[p]$.

In this section, we establish certain formul\ae\ providing relations between zeroes of the functions
$Q[p]$ and $Q[p_{\sigma}]$. We start with the following property of zeroes of the polynomial~
$F[p]$ (see Lemma~2.3 in~\cite{Tyaglov_Hawaii}).

%%%%%%%%%%%%%%%%%%%%%%%%%%%%%%%%%%%%%%%%%%%%%%%%%%%%%%%%%%%%%%%%%%%%%%%%%%%%%%%%%%%%%%%%%%%%%%%%%%%%%%%%%%%%%%%%%%%%%%%%%%%%%%%%%%%%%%%%
\begin{theorem}\label{Theorem.multiple.zero.Q.and.Q1}
Let $\lambda\in\mathbb{C}$ be a zero of $Q[p_{\sigma}]$ of multiplicity $r\geqslant1$ such that
$p(\lambda)\neq0$, $p'(\lambda)\neq0$, $p_{\sigma}(\lambda)\neq0$. If~$Q[p](\lambda)=0$, then
$\lambda$ is a zero of $Q[p]$ of multiplicity $r+1$.
\end{theorem}
%%%%%%%%%%%%%%%%%%%%%%%%%%%%%%%%%%%%%%%%%%%%%%%%%%%%%%%%%%%%%%%%%%%%%%%%%%%%%%%%%%%%%%%%%%%%%%%%%%%%%%%%%%%%%%%%%%%%%%%%%%%%%%%%%%%%%%%%
\begin{proof} Since $p(\lambda)\neq0$, $p_{\sigma}(\lambda)\neq0$, we can deal with $F[p]$ and $F[p_{\sigma}]$
without loss of generality.

First we find a relationship between $F[p]$, $F[p']$ and $F[p_{\sigma}]$. From~\eqref{F.functions} we have
\begin{equation}\label{Theorem.multiple.zero.Q.and.Q1.proof.1}
\begin{array}{c}
p'(z)F[p'](z)=[p'(z)]^2p'''(z)-p'(z)[p''(z)]^2=\\
\\
=p(z)p''(z)p'''(z)-p'(z)[p''(z)]^2-p'''(z)F[p](z)%=\\
%\\
=p''(z)F'[p](z)-p'''(z)F[p](z).
\end{array}
\end{equation}
Now for $F[p_{\sigma}]$ one has
\begin{equation*}
F[p_{\sigma}](z)=p_{\sigma}(z)p_{\sigma}''(z)-[p_{\sigma}'(z)]^2=F[p'](z)+\sigma F'[p](z)+\sigma^2 F[p](z).
\end{equation*}
From this formula and~\eqref{Theorem.multiple.zero.Q.and.Q1.proof.1}, we obtain
\begin{equation}\label{derivatives}
p'(z)F[p_{\sigma}](z)=p'_{\sigma}(z)F'[p](z)+(\sigma p'_\sigma(z)-p''_{\sigma}(z))F[p](z).
\end{equation}

Let $\lambda$ be a zero of $F[p_{\sigma}]$ of multiplicity $r$ and
$F[p](\lambda)=0$. Then $p''(\lambda)\neq0$, since $p'(\lambda)$ would vanish otherwise.
Moreover, $p'_{\sigma}(\lambda)\neq0$ as well. Indeed, from $F[p](\lambda)=0$ and
$p(\lambda)\neq0$, $p'(\lambda)\neq0$, it follows that
\begin{equation}\label{Theorem.multiple.zero.Q.and.Q1.proof.3}
\dfrac{p'(\lambda)}{p(\lambda)}=\dfrac{p''(\lambda)}{p'(\lambda)}.
\end{equation}
If $p'_{\sigma}(\lambda)=0$, then from~\eqref{Theorem.multiple.zero.Q.and.Q1.proof.3} we get
$$
-\sigma=\dfrac{p'(\lambda)}{p(\lambda)}=\dfrac{p''(\lambda)}{p'(\lambda)}
$$
that implies $p_{\sigma}(\lambda)=0$, a contradiction.

Now from formul\ae~\eqref{Theorem.multiple.zero.Q.and.Q1.proof.1}--\eqref{derivatives}
it follows by induction that
\begin{equation}\label{Theorem.multiple.zero.Q.and.Q1.proof.4}
p'(\lambda)F^{(j)}[p_{\sigma}](\lambda)=p'_{\sigma}(\lambda)F^{(j+1)}[p](\lambda),\qquad j=0,1,\ldots,r.
\end{equation}
Since $\lambda$ is a zero of $F[p_{\sigma}](\lambda)$ of multiplicity $r$, one has
from~\eqref{Theorem.multiple.zero.Q.and.Q1.proof.4} that $F^{(k)}[p](\lambda)=0$, $k=0,1,\ldots,r$, and
$F^{(r+1)}[p](\lambda)\neq0$. Hence, $\lambda$ is a zero of $F[p]$ of multiplicity exactly $r+1$. But
$p(\lambda)\neq0$ by assumption, therefore,~$\lambda$~is a~zero of $Q[p]$ of multiplicity $r+1$.
\end{proof}

As a consequence of~\eqref{Theorem.multiple.zero.Q.and.Q1.proof.1}--\eqref{derivatives} we immediately obtain the following.

\begin{corol}\label{Corol.multiple.zero.Q.and.Q.sigma.converse}
Let $\lambda\in\mathbb{C}$ be a zero of $Q[p]$ of multiplicity $r\geqslant2$ and $p(\lambda)\neq0$, $p'(\lambda)\neq0$. Then $\lambda$
is a zero of $Q[p_{\sigma}]$ of multiplicity $r-1$ for any $\sigma\in\mathbb{R}$ such that $p_{\sigma}(\lambda)\neq0$.
\end{corol}

\setcounter{equation}{0}

%%%%%%%%%%%%%%%%%%%%%%%%%%%%%%%%%%%%%%%%%%%%%%%%%%%%%%%%%%%%%%%%%%%%%%%%%%%%%%%%%%
\section{Relations between zeroes of $Q[p]$ and $p_{\sigma}$}\label{section:Zeroes.of.Q.and.p_sigma}
%%%%%%%%%%%%%%%%%%%%%%%%%%%%%%%%%%%%%%%%%%%%%%%%%%%%%%%%%%%%%%%%%%%%%%%%%%%%%%%%%%

It turns out that multiple zeroes of $Q[p]$ have a relation with multiple zeroes of polynomials $p_{\sigma}$ defined in~\eqref{p.sigma}.
%%%%%%%%%%%%%%%%%%%%%%%%%%%%%%%%%%%%%%%%%%%%%%%%%%%%%%%%%%%%%%%%%%%%%%%%%%%%%%%%%%%%%%%%%%%%%%%%%%%%%%%%%%%%%%%%%%%%%%%%%%%%%%%%%%%%%%%%
\begin{theorem}\label{Thm.multiple.zero.Q}
Let $p$ be a real polynomial, and $\lambda\in\mathbb{C}$, $p(\lambda)\neq0$. The number $\lambda$ is a zero of the
function~$Q[p]$ of multiplicity $r\geqslant1$ if and only if there exists $\sigma\in\mathbb{R}$ such that $\lambda$
is a zero of the polynomial $p_{\sigma}$ of multiplicity $r+1$.
\end{theorem}
%%%%%%%%%%%%%%%%%%%%%%%%%%%%%%%%%%%%%%%%%%%%%%%%%%%%%%%%%%%%%%%%%%%%%%%%%%%%%%%%%%%%%%%%%%%%%%%%%%%%%%%%%%%%%%%%%%%%%%%%%%%%%%%%%%%%%%%%
\begin{proof}
Suppose first that $\lambda$ is a zero of multiplicity $r+1$ of $p_{\sigma}$ for some $\sigma\in\mathbb{R}\setminus\{0\}$. Then we have
\begin{equation*}
p_{\sigma}^{(j)}(\lambda)=p^{(j+1)}(\lambda)+\sigma p^{(j)}(\lambda)=0,\qquad j=0,1,\ldots,r,
\end{equation*}
and
\begin{equation*}
p_{\sigma}^{(r+1)}(\lambda)=p^{(r+2)}(\lambda)+\sigma p^{(r+1)}(\lambda)\neq0.
\end{equation*}
Consequently, the following identities hold
\begin{equation}\label{Lemma.multiple.zero.Q.proof.1}
\dfrac{p^{(j+1)}(\lambda)}{p^{(j)}(\lambda)}=\dfrac{p^{(j)}(\lambda)}{p^{(j-1)}(\lambda)}=-\sigma,\qquad j=1,\ldots,r,
\end{equation}
since $p^{(j)}(\lambda)\neq0$, $j=0,1,\ldots,r+1$ due to $\sigma\neq0$ by assumption.
Now identities~\eqref{Lemma.multiple.zero.Q.proof.1} imply
\begin{equation}\label{Lemma.multiple.zero.Q.proof.2}
Q\left[p^{(j-1)}\right](\lambda)=0,\qquad j=1,\ldots,r.
\end{equation}

Moreover, since $p_{\sigma}^{(r+1)}(\lambda)\neq0$, one has
\begin{equation*}
\dfrac{p^{(r+2)}(\lambda)}{p^{(r+1)}(\lambda)}\neq-\sigma.
\end{equation*}
Therefore,
\begin{equation}\label{Lemma.multiple.zero.Q.proof.3}
F'\left[p^{(r-1)}\right](\lambda)=p^{(r-1)}(\lambda)p^{(r+2)}(\lambda)-p^{(r)}(\lambda)p^{(r+1)}(\lambda)\neq0,
\end{equation}
otherwise, we would have
\begin{equation*}
\dfrac{p^{(r+2)}(\lambda)}{p^{(r+1)}(\lambda)}=\dfrac{p^{(r)}(\lambda)}{p^{(r-1)}(\lambda)}=-\sigma,
\end{equation*}
a contradiction. Thus, inequality~\eqref{Lemma.multiple.zero.Q.proof.3} implies
\begin{equation*}
Q'\left[p^{(r-1)}\right](\lambda)=\dfrac{p^{(r-1)}(\lambda)p^{(r+2)}(\lambda)-p^{(r)}(\lambda)p^{(r+1)}(\lambda)}{\left[p^{(r-1)}(\lambda)\right]^2}\neq0.
\end{equation*}

Consequently, $\lambda$ is a simple zero of $Q[p^{(r-1)}]$. Now from~\eqref{Lemma.multiple.zero.Q.proof.2} and
Theorem~\ref{Theorem.multiple.zero.Q.and.Q1} it follows that $\lambda$ is a zero of $Q[p^{(r-j)}]$ of multiplicity $j$,
$j=1,\ldots,r$. That is, $\lambda$ is a zero of $Q[p]$ of multiplicity $r$:
\begin{equation}\label{Lemma.multiple.zero.Q.proof.4}
Q^{(j)}\left[p\right](\lambda)=0,\qquad j=0,1,\ldots,r-1,
\end{equation}
as required.

If now $\sigma=0$, that is, if $p^{(j)}(\lambda)=0$, $j=1,\ldots,r+1$, and $p^{(r+2)}(\lambda)\neq0$, then
from~\eqref{main.function.2} it follows that $\lambda$ is a zero of $Q[p]$ of multiplicity $r$,
so identities~\eqref{Lemma.multiple.zero.Q.proof.4} hold in this case, as well.

\vspace{3mm}

Conversely, let $\lambda$ be a zero of $Q[p]$ of multiplicity $r$, so identities~\eqref{Lemma.multiple.zero.Q.proof.4} hold.
Since $p(\lambda)\neq0$ by assumption, these identities can be rewritten as follows
\begin{equation*}
F^{(j)}\left[p\right](\lambda)=0,\qquad j=0,1,\ldots,r-1.
\end{equation*}
If $p'(\lambda)=0$, then from~\eqref{main.function.2} we have $p^{(j)}(\lambda)=0$, $j=1,\ldots,r+1$,
and $p^{(r+2)}(\lambda)\neq0$, so the statement of the theorem holds in this case.

Let now $p'(\lambda)\neq0$. Then the identity $F[p](\lambda)=0$ implies
\begin{equation*}
\dfrac{p'(\lambda)}{p(\lambda)}=\dfrac{p''(\lambda)}{p'(\lambda)}.
\end{equation*}
So we can introduce the number
\begin{equation*}
\sigma\stackrel{def}{=}-\dfrac{p'(\lambda)}{p(\lambda)}=-\dfrac{p''(\lambda)}{p'(\lambda)}.
\end{equation*}
Thus, we get
\begin{equation*}
p_{\sigma}(\lambda)=p'(\lambda)+\sigma p(\lambda)=0,
\end{equation*}
and
\begin{equation*}
p'_{\sigma}(\lambda)=p''(\lambda)+\sigma p'(\lambda)=0.
\end{equation*}

Now formula~\eqref{derivatives} yields
\begin{equation*}
F^{(j)}[p']=0,\qquad j=0,1,\ldots,r-2.
\end{equation*}
In particular,
\begin{equation*}
F[p'](\lambda)=p'(\lambda)p'''(\lambda)-\left[p''(\lambda)\right]^2=0,
\end{equation*}
so
\begin{equation*}
\dfrac{p'''(\lambda)}{p''(\lambda)}=\dfrac{p''(\lambda)}{p'(\lambda)}=-\sigma,
\end{equation*}
and therefore,
\begin{equation*}
p''_{\sigma}(\lambda)=p'''(\lambda)+\sigma p''(\lambda)=0.
\end{equation*}

Now by induction we obtain the identities
\begin{equation}\label{Lemma.multiple.zero.Q.proof.5}
F^{(j)}\left[p^{(k)}\right](\lambda)=0,\qquad k=0,1,\ldots,r-1, \ j=0,1,\ldots,r-k-1,
\end{equation}
which give
\begin{equation*}
p^{(k+1)}_{\sigma}(\lambda)=p^{(k+2)}(\lambda)+\sigma p^{(k+1)}(\lambda)=0,\qquad k=0,1,\ldots,r-1,
\end{equation*}
so
\begin{equation*}
p^{(r)}_{\sigma}(\lambda)=p^{(r+1)}(\lambda)+\sigma p^{(r)}(\lambda)=0.
\end{equation*}

Since $\lambda$ is a zero of $Q[p]$ of multiplicity $r$, it is a simple zero of $F\left[p^{(r-1)}\right]$ by~\eqref{Lemma.multiple.zero.Q.proof.5},
that is,
\begin{equation*}
F'\left[p^{(r-1)}\right](\lambda)=p^{(r-1)}(\lambda)p^{(r+2)}(\lambda)-p^{(r)}(\lambda)p^{(r+1)}(\lambda)\neq0,
\end{equation*}
so
\begin{equation*}
\dfrac{p^{(r+2)}(\lambda)}{p^{(r+1)}(\lambda)}\neq\dfrac{p^{(r+1)}(\lambda)}{p^{(r)}(\lambda)}=-\sigma.
\end{equation*}
%\
Thus, $p^{(r+1)}_{\sigma}(\lambda)=p^{(r+2)}(\lambda)+\sigma p^{(r+1)}(\lambda)\neq0$, as required.
\end{proof}

As a by-product, we obtain the following curious fact.
\begin{corol}\label{Corol.finite.number.multiple roots}
Let $p$ be a real polynomial with simple zeroes. Then the number of complex points that
can be multiple zeroes of the polynomials of the family~$p_{\sigma}$
does not exceed $2\deg p-2$, and the points do not depend on $\sigma$. Moreover, the number
of real such points is bounded by $Z_{\mathbb{C}}(p)$.
\end{corol}
\begin{proof}
Since $p$ has no multiple zeroes, any zero of $F[p](z)$ is the zero of $Q[p](z)$. Moreover, $\deg F[p]=2\deg p-2$.
Now by Theorem~\ref{Thm.multiple.zero.Q}, we obtain that any multiple zero of a polynomial $p_{\sigma}$ for
some $\sigma$ is a zero of $F[p]$. The second statement of the corollary follows from the inequality~\eqref{Hawaii.conjecture} and
Theorem~\ref{Thm.multiple.zero.Q}.
\end{proof}

\setcounter{equation}{0}

%%%%%%%%%%%%%%%%%%%%%%%%%%%%%%%%%%%%%%%%%%%%%%%%%%%%%%%%%%%%%%%%%%%%%%%%%%%%%
\section{Cauchy indices}\label{section:Cauchy.indices}
%%%%%%%%%%%%%%%%%%%%%%%%%%%%%%%%%%%%%%%%%%%%%%%%%%%%%%%%%%%%%%%%%%%%%%%%%%%%%

In this section, we establish a generalisation (one direction)
of~\cite[Theorem~1 and Corollary~1]{CravenCsordasSmith} using the Cauchy indices
whose definition and basic properties we present first.

%In this section, we introduce a special counter known as the Cauchy index.

Consider a real rational function $R$ which may  have a pole at $\infty$.
\begin{definition}\label{Def.Cauchy.index.at.pole}
The quantity
\begin{equation}\label{Cauchy.index.at.odd.pole}
\Ind\nolimits_{\omega}(R)\eqbd
\begin{cases}
   \; +1&\;\text{if}\quad R(\omega-0)<0<R(\omega+0),\\
   \; -1&\;\text{if}\quad R(\omega-0)>0>R(\omega+0),
\end{cases}
\end{equation}
is called the \textit{index\/} of the function $R$ at its
\textit{real\/} pole $\omega$ of \textit{odd\/} order.

We also set
\begin{equation}\label{Cauchy.index.at.even.pole}
\Ind\nolimits_{\omega}(R)\eqbd 0
\end{equation}
if $\omega$ is a real pole of the function $R$ of \textit{even\/}
order.
\end{definition}

Suppose that the function $R$ has $m$ real poles in total, viz.,
$\omega_1<\omega_2<\dots<\omega_m$.

%
%\smallskip

\begin{definition}\label{Def.Cauchy.index.on.interval}
The quantity
\begin{equation}\label{Cauchy.index.on.interval}
\Ind\nolimits_a^b(R)\eqbd \sum\limits_{i\,\colon\,a<\omega_i<b}
\Ind\nolimits_{\omega_i}(R).
\end{equation}
is called the \textit{Cauchy index} of the function $R$ on the
interval~$(a,b)$.
\end{definition}

We are primarily interested in the quantity $\IndC(R)$, the Cauchy
index of $R$ on the real line. However, since the function $R$
may have a pole at the point $\infty$, it is convenient for us to
consider this pole as \textit{real}. From this point of view, let
us introduce the index at $\infty$ following, e.g.,~\cite{Barkovsky.2,Holtz_Tyaglov}.
To do so, we consider the function~$R$ as a map on the \textit{projective
line} $\PR^1\eqbd \mathbb{R}^1\cup\{\infty\}$ into itself. So, if the
function $R$ has a pole at~$\infty$, then we let
\begin{equation}\label{Cauchy.index.at.infty}
\IndI(R) \eqbd
\begin{cases}
\; +1 & \quad\text{if}\quad R(+\infty)<0<R(-\infty),\\
\; -1& \quad\text{if}\quad R(+\infty)>0>R(-\infty),\\
\; \quad\! 0 & \quad\text{if}\quad \sgn R(+\infty)=\sgn R(-\infty).
\end{cases}
\end{equation}
Thus, the generalised Cauchy  index of the function $R$ on the
projective real line is
\begin{equation}\label{Cauchy.index.on.projective.line}
\IndPR(R)\eqbd \IndC(R)+\IndI(R).
\end{equation}

\smallskip \begin{remark}\label{remark.2.1}
Obviously, if the function $R$ has no pole at $\infty$, then the generalised
Cauchy index  $\IndPR(R)$ coincides with the usual Cauchy index $\IndC(R)$.
\end{remark}

Following~\cite{Barkovsky.2}, we list a few properties of
generalised Cauchy indices, which will be of use later.
First, note that a polynomial $q(z)=cz^{\nu}+\cdots$
$(\nu=\deg q)$ can be viewed as a rational function with
a single pole at $\infty$, hence
\begin{equation}\label{Cauchy.index.of.poly}
\IndI(q)=
\begin{cases}
-\sgn c\; &\text{if}\quad\nu\quad\text{is odd},\\
\;\;\; 0     \; &\text{if}\quad\nu\quad\text{is even}.\\
\end{cases}
\end{equation}

The following theorem (see~\cite{Barkovsky.2,Holtz_Tyaglov} and references there)
collects all properties of Cauchy indices that we need.
\begin{theorem}\label{Th.Cauchy.index.properties}
Let $R$ be a real rational function.
\begin{itemize}
\item[1)] If $d$ is a real constant, then $\IndPR(d+R)=\IndPR(R)$.
\item[2)] If $q$ is a real polynomial and $|R(\infty)|<\infty$,
then $\IndPR(q+R)=\IndC(R)+\IndI(q)$.
\item[3)] If $R_1$ and $R_2$ are real rational functions that have no
real poles in common, then
\begin{equation*}%\label{Cauchy.index.of.functions.sum}
\IndC(R_1+R_2)=\IndC(R_1)+\IndC(R_2).
\end{equation*}
\item[4)] $\IndPR\left(-\dfrac1R\right)=\IndPR(R)$.
\end{itemize}
\end{theorem}
%

%Any rational function $G$ can be represented as the following
%continued fraction
%%
%\begin{equation}\label{continued.fraction.general.via.Sturm}
%G(z)=q_0(z)+\dfrac1{q_1(z)-\cfrac1{q_2(z)-\cfrac1{q_3(z)-\cfrac1{\ddots-\cfrac1{q_k(z)}}}}}
%\end{equation}
%
%Here the polynomials $q_i$ have the form
%
%\begin{equation*}%\label{Sturm.cont.frac.quotients}
%q_j(z)=\alpha_jz^{n_j}+\cdots,\quad\;\;\;  \alpha_j\neq 0,\qquad
%j=0,1,2,\ldots,k,
%\end{equation*}
%
%with $n_1+n_2+\cdots+n_k=r$
%and $n_i\geq 1$, $i=1,\ldots,k$, where $r$ is the number of poles of the function $G$ counted
%with multiplicities, and $q_0$ can be a (zero) constant.
%Then the following result holds:
%
%\smallskip \begin{theorem}[\cite{Barkovsky.2,Holtz_Tyaglov}]\label{Th.index.via.quotients}
%If a rational function $G$ is represented by a continued
%fraction~\eqref{continued.fraction.general.via.Sturm}, then
%
%\begin{equation*}%\label{index.via.quotients}
%\IndPR(G)=\IndI(q_0)-\sum_{j=1}^{k}\IndI(q_j).
%\end{equation*}
%
%\end{theorem}
%\smallskip

Now we are in a position to establish an extension of an important result by T.\,Craven, G.\,Csordas
and W.\,Smith~\cite[Theorem~1 and Corollary~1]{CravenCsordasSmith} for polynomials (see also~\cite[Problem~133]{GunterKuzmin}
and~\cite[Problem~728]{FaddeevSominsky}).

\begin{theorem}\label{Theorem.real.zeroes.p.sigma}
Let $p$ be a real polynomial with $Z_{\mathbb{C}}(p)=2m$. If there exists $\sigma\in\mathbb{R}$ such that
$Q[p_{\sigma}](x)<0$ for all $x\in\mathbb{R}$ where it is defined, then
\begin{equation}\label{Number.real.zeroes.Q.ext}
Z_{\mathbb{R}}\left(Q[p]\right)=2m-2m_{\sigma},
\end{equation}
where $2m_{\sigma}=Z_{\mathbb{C}}\left(p_{\sigma}\right)$. Moreover, the real zeroes of $Q[p]$ different
from zeroes of $p_{\sigma}$ of multiplicity greater than $2$ (if any) are simple.
\end{theorem}
Note that $2m_{\sigma}\leqslant2m$ by~\eqref{Propos.sigma.p.zeroes}.
\begin{proof}
%Without loss of generality, we suppose that $p$ has only simple zeroes.

Suppose first that $\sigma\neq0$, so $\deg p_{\sigma}=\deg p=:n$. By assumption
\begin{equation}\label{Q1.negativity}
Q[p_{\sigma}](x)=\dfrac{p_{\sigma}(x)p''_{\sigma}(x)-[p'_{\sigma}(x)]^2}{p^2_{\sigma}(x)}<0
\end{equation}
for any $x\in\mathbb{R}$ such that $p_{\sigma}(x)\neq0$.

Consider the function
\begin{equation*}
H(z)=\dfrac{p_{\sigma}(z)}{p(z)}=\dfrac{p'(z)}{p(z)}+\sigma.
\end{equation*}
It is easy to see that
\begin{equation}\label{H.derivartive.1}
H'(z)=Q[p](z)=\dfrac{p(z)p''(z)-[p'(z)]^2}{p^2(z)},
\end{equation}
and
\begin{equation}\label{H.log.derivartive.1}
\dfrac{H'(z)}{H(z)}=\dfrac{p(z)p''(z)-[p'(z)]^2}{p(z)p_{\sigma}(z)}=\dfrac{p(z)p'_{\sigma}(z)-p'(z)p_{\sigma}(z)}{p(z)p_{\sigma}(z)}=\dfrac{p'_{\sigma}(z)}{p_{\sigma}(z)}-\dfrac{p'(z)}{p(z)}.
\end{equation}
Since $\dfrac{H'(z)}{H(z)}\to0$ as $|z|\to\infty$, from~\eqref{Cauchy.index.on.projective.line} we have
\begin{equation}\label{Theorem.real.zeroes.p.sigma.proof.0}
\IndPR\left(\dfrac{H'}{H}\right)=\IndC\left(\dfrac{H'}{H}\right)%=\IndC\left(\dfrac{p'_{\sigma}}{p_{\sigma}}\right)-\IndC\left(\dfrac{p'}{p}\right).
\end{equation}

Suppose that $p(z)$ has $d_1$ simple real zeroes and $d_2$ multiple real zeroes:
\begin{equation*}%\label{Theorem.real.zeroes.p.sigma.proof.0.0.1}
p(z)=a\prod\limits_{i=1}^{d_1}(z-\lambda_i)\prod\limits_{j=1}^{d_2}(z-\mu_j)^{n_j}q(z),
\end{equation*}
where $a>0$, $n_j\geqslant2$, $j=1,\ldots,d_2$, and $q(x)\neq0$ for $x\in\mathbb{R}$. Here $\deg q=2m$. Then
\begin{equation}\label{Theorem.real.zeroes.p.sigma.proof.0.1}
p_{\sigma}(z)=a\sigma\prod\limits_{j=1}^{d_2}(z-\mu_j)^{n_j-1}g(z)h(z),
\end{equation}
where $g(z)$ has only real zeroes by assumption, and $h(x)\neq0$ for $x\in\mathbb{R}$, $\deg h=2m_{\sigma}$. So we have
\begin{equation}\label{Theorem.real.zeroes.p.sigma.proof.1}
\deg g=n-\sum\limits_{j=1}^{d_2}(n_j-1)-2m_{\sigma}=d_1+d_2+\sum\limits_{j=1}^{d_2}(n_j-1)+2m-\sum\limits_{j=1}^{d_2}(n_j-1)-2m_{\sigma}=d_1+d_2+2m-2m_{\sigma}.
\end{equation}
Let
\begin{equation}\label{Theorem.real.zeroes.p.sigma.proof.2}
g(z)=\prod\limits_{k=1}^{l_1}(z-t_k)\prod\limits_{i=1}^{l_2}(z-s_i)^{r_i},
\end{equation}
where $r_i\geqslant2$, $i=1,\ldots,l_2$, and the numbers $t_k$ and $s_i$ are real and different from $\mu_j$.
Then we have
\begin{equation*}
\dfrac{p'_{\sigma}(z)}{p_{\sigma}(z)}=\sum\limits_{j=1}^{d_2}\dfrac{n_j-1}{z-\mu_j}+\sum\limits_{k=1}^{l_1}\dfrac{1}{z-t_k}+
\sum\limits_{i=1}^{l_2}\dfrac{r_i}{z-s_i}+\dfrac{h'(z)}{h(z)},
\end{equation*}
and
\begin{equation*}
\dfrac{p'(z)}{p(z)}=\sum\limits_{i=1}^{d_1}\dfrac{1}{z-\lambda_i}+\sum\limits_{j=1}^{d_2}\dfrac{n_j}{z-\mu_j}+\dfrac{q'(z)}{q(z)}.
\end{equation*}
Consequently, from~\eqref{H.log.derivartive.1} one gets
\begin{equation}\label{H.log.derivartive.2}
\dfrac{H'(z)}{H(z)}=\sum\limits_{k=1}^{l_1}\dfrac{1}{z-t_k}+\sum\limits_{i=1}^{l_2}\dfrac{r_i}{z-s_i}-
\sum\limits_{i=1}^{d_1}\dfrac{1}{z-\lambda_i}-\sum\limits_{j=1}^{d_2}\dfrac{1}{z-\mu_j}+\dfrac{h'(z)}{h(z)}-\dfrac{q'(z)}{q(z)}.
\end{equation}
The polynomials $h$ and $q$ have no real zeroes, so their logarithmic derivatives do not contribute to $\IndC\left(\dfrac{H'}{H}\right)$. Thus,
from Theorem~\ref{Th.Cauchy.index.properties} and from~\eqref{Theorem.real.zeroes.p.sigma.proof.1}--\eqref{H.log.derivartive.2} we obtain
\begin{equation}\label{Theorem.real.zeroes.p.sigma.proof.2.1}
\IndC\left(\dfrac{H'}{H}\right)=l_1+l_2-d_1-d_2=d_1+d_2+2m-2m_{\sigma}-\sum\limits_{i=1}^{l_2}(r_i-1)-d_1-d_2=2m-2m_{\sigma}-\sum\limits_{i=1}^{l_2}(r_i-1).
\end{equation}
Note that $s_i$, $i=1,\ldots,l_2$, is a zero of $Q[p]$ of multiplicity $r_i-1$ by Theorem~\ref{Thm.multiple.zero.Q}.

Furthermore, since
\begin{equation*}
-\dfrac{H(z)}{H'(z)}=Az^2+Bz+C+\sum\limits_{k=0}^{\infty}\dfrac{c_k}{z^k},
\end{equation*}
it follows that $\IndI\left(-\dfrac{H}{H'}\right)=0$. Now from Theorem~\ref{Th.Cauchy.index.properties}, \eqref{Cauchy.index.on.projective.line},
and~\eqref{Theorem.real.zeroes.p.sigma.proof.0} one has
\begin{equation}\label{H/H'.index.1}
\IndC\left(-\dfrac{H}{H'}\right)=\IndPR\left(-\dfrac{H}{H'}\right)=\IndPR\dfrac{H'}{H}=\IndC\dfrac{H'}{H}=2m-2m_{\sigma}-\sum\limits_{i=1}^{l_2}(r_i-1).
\end{equation}

Note that the numbers $s_i$, $i=1,\ldots,l_2$, are not poles of the function $-\dfrac{H}{H'}$, but they are simple
zeroes of $-\dfrac{H}{H'}$.

Let now $\lambda\in\mathbb{R}$ be a pole of $-\dfrac{H}{H'}$. Clearly, $p(\lambda)\neq0$ and $p_{\sigma}(\lambda)\neq0$, and by~\eqref{H.log.derivartive.1} we have
\begin{equation*}%\label{Theorem.real.zeroes.p.sigma.proof.3}
F[p](\lambda)=p(\lambda)p''(\lambda)-[p'(\lambda)]^2=0.
\end{equation*}

\vspace{2mm}

\noindent \textsc{I.}  If $p'(\lambda)=0$, then $p''(\lambda)=0$, so $p'_{\sigma}(\lambda)=0$. At the same time, $p'''(\lambda)\neq0$, since otherwise, we would
have $p''_{\sigma}(\lambda)=0$ that contradicts~\eqref{Q1.negativity}. Thus, in this case
\begin{equation*}
F'[p](\lambda)=p(\lambda)p'''(\lambda)-p'(\lambda)p''(\lambda)\neq0,
\end{equation*}
so $\lambda$ is a simple pole of $-\dfrac{H}{H'}$ (and a simple zero of $Q[p]$), and
\begin{equation}\label{Theorem.real.zeroes.p.sigma.proof.4}
-\dfrac{H(z)}{H'(z)}=\dfrac{A_{\lambda}}{z-\lambda}+O(1)\qquad\text{as}\quad z\to\lambda,
\end{equation}
where by~\eqref{Q1.negativity} and~\eqref{H.log.derivartive.1}, one has
\begin{equation}\label{Theorem.real.zeroes.p.sigma.proof.5}
%\begin{array}{l}
A_{\lambda}=\lim\limits_{z\to\lambda}\dfrac{p(z)p_{\sigma}(z)(z-\lambda)}{p'(z)p_{\sigma}(z)-p(z)p'_{\sigma}(z)}
=\dfrac{p(\lambda)p_{\sigma}(\lambda)}{p''(\lambda)p_{\sigma}(\lambda)-p(\lambda)p''_{\sigma}(\lambda)}=-\dfrac{p_{\sigma}(\lambda)}{p_{\sigma}''(\lambda)}>0,
\end{equation}
since $p'_{\sigma}(\lambda)=0$ as we showed above. Thus, \eqref{Theorem.real.zeroes.p.sigma.proof.4}--\eqref{Theorem.real.zeroes.p.sigma.proof.5} imply
\begin{equation}\label{Theorem.real.zeroes.p.sigma.proof.6}
\Ind\nolimits_{\lambda}\left(-\dfrac{H}{H'}\right)=1,
\end{equation}
according to~\eqref{Cauchy.index.at.odd.pole}.

\vspace{2mm}

\noindent\textsc{II.} Suppose now that $p'(\lambda)\neq0$. Let us show that this is a simple pole of $-\dfrac{H}{H'}$ (a simple zero of $Q[p]$), as well.
Indeed, from~\eqref{H.log.derivartive.1} we have
\begin{equation*}
F[p](\lambda)=p(\lambda)p''(\lambda)-[p'(\lambda)]^2=p(\lambda)p'_{\sigma}(\lambda)-p'(\lambda)p_{\sigma}(\lambda)=0,
\end{equation*}
so
\begin{equation}\label{Theorem.real.zeroes.p.sigma.proof.7}
\dfrac{p'_{\sigma}(\lambda)}{p_{\sigma}(\lambda)}=\dfrac{p'(\lambda)}{p(\lambda)}.
\end{equation}
At the same time, by~\eqref{Q1.negativity} and \eqref{Theorem.real.zeroes.p.sigma.proof.7}
\begin{equation}\label{Theorem.real.zeroes.p.sigma.proof.8}
\begin{array}{c}
\dfrac{F'[p](\lambda)}{p(\lambda)p_{\sigma}(\lambda)}=\dfrac{p(\lambda)p''_{\sigma}(\lambda)-p''(\lambda)p_{\sigma}(\lambda)}{p(\lambda)p_{\sigma}(\lambda)}=
\dfrac{p''_{\sigma}(\lambda)}{p_{\sigma}(\lambda)}-\dfrac{p''(\lambda)}{p(\lambda)}<
\left(\dfrac{p'_{\sigma}(\lambda)}{p_{\sigma}(\lambda)}\right)^2-\dfrac{p''(\lambda)}{p(\lambda)}=\\
\\
=\left(\dfrac{p'(\lambda)}{p(\lambda)}\right)^2-\dfrac{p''(\lambda)}{p(\lambda)}=\dfrac{F[p](\lambda)}{p^2(\lambda)}=0.
\end{array}
\end{equation}
Thus, $\lambda$ is a simple pole of $-\dfrac{H}{H'}$, and from~\eqref{Theorem.real.zeroes.p.sigma.proof.8} we have
\begin{equation*}
-\dfrac{H(z)}{H'(z)}=\dfrac{B_{\lambda}}{z-\lambda}+O(1)\qquad\text{as}\quad z\to\lambda,
\end{equation*}
with
\begin{equation*}
B_{\lambda}=\lim\limits_{z\to\lambda}\dfrac{p(z)p_{\sigma}(z)(z-\lambda)}{-F[p](z)}=-\dfrac{p(\lambda)p_{\sigma}(\lambda)}{F'[p](\lambda)}>0,
\end{equation*}
so~\eqref{Theorem.real.zeroes.p.sigma.proof.6} holds.

Thus, we obtained that all the real poles of $-\dfrac{H}{H'}$ are simple, and the Cauchy index of $-\dfrac{H}{H'}$ at every
real pole equals $1$. From~\eqref{H/H'.index.1} it follows now that the function $-\dfrac{H}{H'}$ has exactly
\begin{equation}\label{Theorem.real.zeroes.p.sigma.proof.9}
2m-2m_{\sigma}-\sum\limits_{i=1}^{l_2}(r_i-1)
\end{equation}
real poles, counting multiplicities. All these poles are real simple zeroes of the function~$Q[p]$,
according to~\eqref{H.derivartive.1}--\eqref{H.log.derivartive.1}. Additionally, the function $Q[p]$
has real zeroes at the points $s_i$, $i=1,\ldots,l_2$, and the number $s_i$ is a zero of $Q[p]$ of
multiplicity $r_i-1$. In particular, it is a multiple zero of $Q[p]$ if $r_i\geqslant3$. As we mentioned
above, at the points $s_i$, $i=1,\ldots,l_2$, the function $Q[p]$ has exactly
\begin{equation*}
\sum\limits_{i=1}^{l_2}(r_i-1)
\end{equation*}
real zeros, counting multiplicities, which together with~\eqref{Theorem.real.zeroes.p.sigma.proof.9}
gives~\eqref{Number.real.zeroes.Q.ext}.

\vspace{5mm}

Let now $\sigma=0$. Then the formula~\eqref{Number.real.zeroes.Q.ext} can be proved in a similar way.
The only difference is that for $\sigma=0$, $p_0=p'+0\cdot p=p'$, and $\deg p_0=n-1$. In this case,
the index of the logarithmic derivative of~$H$ on the real line is
\begin{equation*}
\IndC\left(\dfrac{H'}{H}\right)=2m-2m_{\sigma}-1-\sum\limits_{i=1}^{l_2}(r_i-1),
\end{equation*}
see~\eqref{Theorem.real.zeroes.p.sigma.proof.2.1}, and
\begin{equation*}
-\dfrac{H(z)}{H'(z)}=z+c_0+\sum\limits_{k=0}^{\infty}\dfrac{c_k}{z^k},
\end{equation*}
so it follows that $\IndI\left(-\dfrac{H}{H'}\right)=-1$, and from Theorem~\ref{Th.Cauchy.index.properties}, \eqref{Cauchy.index.on.projective.line},
and~\eqref{Theorem.real.zeroes.p.sigma.proof.0} we have
\begin{equation*}%\label{H/H'.index.2}
\begin{array}{l}
\IndC\left(-\dfrac{H}{H'}\right)-1=\IndC\left(-\dfrac{H}{H'}\right)+\IndI\left(-\dfrac{H}{H'}\right)=\IndPR\left(-\dfrac{H}{H'}\right)=\\
\\
=\IndPR\left(\dfrac{H'}{H}\right)=\IndC\left(\dfrac{H'}{H}\right)=2m-2m_{\sigma}-1-\sum\limits_{i=1}^{l_2}(r_i-1).
\end{array}
\end{equation*}
Therefore,
\begin{equation*}
\IndC\left(-\dfrac{H}{H'}\right)=2m-2m_{\sigma}-\sum\limits_{i=1}^{l_2}(r_i-1),
\end{equation*}
where $r_i$ and $l_2$ are defined in~\eqref{Theorem.real.zeroes.p.sigma.proof.0.1} and~\eqref{Theorem.real.zeroes.p.sigma.proof.2} with $p_0=p'$.
Consequently, the numbers $s_i$ are zeroes of~$p'$, so if $\lambda\in\mathbb{R}$ is a pole of $-\dfrac{H}{H'}$, then~$p'(\lambda)\neq0$.
In a similar way as above, one can show that every real pole of $-\dfrac{H}{H'}$ is simple, and the index of $-\dfrac{H}{H'}$
at each real pole is $1$. Thus, formula~\eqref{Number.real.zeroes.Q.ext} can be proved in the same way as for the case $\sigma\neq0$.
\end{proof}
\smallskip

As a by-product, from the proof of Theorem~\ref{Theorem.real.zeroes.p.sigma} we easily obtain the following fact.
\begin{theorem}\label{Theorem.real.zeroes.Q.lower.bound}
Let $p$ be a real polynomial with $Z_{\mathbb{C}}(p)=2m$, and let
\begin{equation*}
2m_{\sigma^{*}}=\min\limits_{\sigma\in\mathbb{R}}\left(Z_{\mathbb{C}}(p_\sigma)\right),
\end{equation*}
where the polynomial $p_{\sigma}$ is defined in~\eqref{p.sigma}. Then
\begin{equation}\label{Lower.bound.number.real.zeroes.Q}
Z_{\mathbb{R}}\left(Q[p]\right)\geqslant2m-2m_{\sigma^{*}}.
\end{equation}
\end{theorem}
\begin{proof}
Indeed, from the proof of Theorem~\ref{Theorem.real.zeroes.p.sigma} it follows that for
arbitrary $\sigma\in\mathbb{R}$
\begin{equation}\label{Theorem.real.zeroes.Q.lowert.bound.proof.1}
\IndC\left(-\dfrac{H}{H'}\right)=2m-2m_{\sigma}-\sum\limits_{i=1}^{l_2}(r_i-1),
\end{equation}
where $r_i$ are multiplicities of multiple zeroes $s_i$ of $p_{\sigma}$ which are not
zeroes of $p$. The poles of $-\dfrac{H}{H'}$ are zeroes of $Q[p]$. Moreover, according to
Definition~\ref{Def.Cauchy.index.at.pole}, the zeroes of $Q[p]$ of even multiplicity do
not contribute to the index~\eqref{Theorem.real.zeroes.Q.lowert.bound.proof.1}. Thus,
the index~\eqref{Theorem.real.zeroes.Q.lowert.bound.proof.1} is equal to the sum of
\textit{distinct} zeroes of $Q[p]$ of odd multiplicities with indices
of various signs (if any) at them. Consequently, we have
\begin{equation}\label{Theorem.real.zeroes.Q.lowert.bound.proof.2}
Z_{\mathbb{R}}\left(Q[p]\right)\geqslant2m-Z_{\mathbb{C}}(p_\sigma)-\sum\limits_{i=1}^{l_2}(r_i-1).
\end{equation}
Now recall that each zero $s_i$ of $p_{\sigma}$ of multiplicity $r_i$ is a zero of $Q[p]$
of multiplicity $r_i-1$, $i=1,\ldots,l_2$. Therefore, additionally to the zeroes guaranteed
by the index~\eqref{Theorem.real.zeroes.Q.lowert.bound.proof.1}, the function $Q[p]$ has
exactly
\begin{equation*}
\sum\limits_{i=1}^{l_2}(r_i-1)
\end{equation*}
zeroes, counting multiplicities, at the points $s_i$, $i=1,\ldots,l_2$. Now
from~\eqref{Theorem.real.zeroes.Q.lowert.bound.proof.2} it follows that
\begin{equation*}\label{Theorem.real.zeroes.Q.lowert.bound.proof.3}
Z_{\mathbb{R}}\left(Q[p]\right)\geqslant2m-Z_{\mathbb{C}}(p_\sigma)
\end{equation*}
for any $\sigma\in\mathbb{R}$. Since $2m-Z_{\mathbb{C}}(p_\sigma)\leqslant2m-2m_{\sigma^{*}}$, the lower bound~\eqref{Lower.bound.number.real.zeroes.Q} is true.
\end{proof}

It is interesting to find out if it is possible to establish Theorem~\ref{Theorem.Hawaii.base.ineq} using the Cauchy indices. However, we believe that a stronger fact is true.
\begin{conjecture}
Let $p$ be a real polynomial with $Z_{\mathbb{C}}(p)=2m>0$. Then there exists $\sigma\in\mathbb{R}$ such that
\begin{equation*}\label{Number.real.zeroes.Q.eq.Q1}
Z_{\mathbb{R}}\left(Q[p]\right)=2m-2m_{\sigma}+Z_{\mathbb{R}}\left(Q[p_{\sigma}]\right),
\end{equation*}
where $2m_{\sigma}=Z_{\mathbb{C}}\left(p_{\sigma}\right)\leqslant2m$, and the polynomial $p_{\sigma}$ is defined in~\eqref{p.sigma}. Moreover, if there exists $x,\sigma\in\mathbb{R}$ with
$Q[p_{\sigma}](x)>0$, then $\sigma$ can be chosen such that $2m_{\sigma}<2m$.

%Moreover, the real zeroes of $Q[p]$ different from zeroes of $p_{\sigma}$ and $Q[p_{\sigma}]$ of multiplicity greater than $2$ (if any) are simple.
\end{conjecture}

%However, we are not sure that in this conjecture $\sigma$ can always be chosen such that $2m_{\sigma}=\min\limits_{\sigma\in\mathbb{R}}\left(Z_{\mathbb{C}}(p_\sigma)\right)$.
%Seems to be not.

%\begin{remark}
%Note that the upper bound in the Theorem~\ref{Theorem.Hawaii.base.ineq} cannot be improved to
%identity, in general. Indeed, calculations show that for the polynomial
%
%\begin{equation*}
%p(z)=(z-3)(z+5)(z+1)(z-15)^3(z-10)(z+7+i)(z+7-i)(z+13+i)(z+13-i)(z+10+i)^5(z+10-i)^5
%\end{equation*}
%
%\end{remark}

\setcounter{equation}{0}
%%%%%%%%%%%%%%%%%%%%%%%%%%%%%%%%%%%%%%%%%%%%%%%%%%%%%%%%%%%%%%%%%%%%%%%%%%%%%
\section*{Acknowledgments}
%%%%%%%%%%%%%%%%%%%%%%%%%%%%%%%%%%%%%%%%%%%%%%%%%%%%%%%%%%%%%%%%%%%%%%%%%%%%%

%The work of M.\,Tyaglov was supported with the support of the Russian Science Foundation Grant 19-71-30002.

The work of M.\,Tyaglov was financially supported by a grant of the Ministry
of Education and Science of the Russian Federation for the establishment and
development of the Leonhard Euler International Mathematical Institute in Saint
Petersburg, agreements 075-15-2025-343.

\end{document}